\documentclass[11pt,a4paper]{article}
\usepackage[utf8]{inputenc}
\usepackage[T1]{fontenc}
\usepackage{lmodern}
\usepackage{amsmath, amssymb, amsthm}
\usepackage[margin=1in]{geometry}
\usepackage{booktabs}
\usepackage{hyperref}
\usepackage{cite}
\usepackage{enumitem}
\usepackage{listings}
\usepackage{xcolor}
\hypersetup{colorlinks=true, linkcolor=blue, citecolor=blue, urlcolor=blue}

\newtheorem{theorem}{Theorem}

\newtheorem{corollary}[theorem]{Corollary}
\newtheorem{proposition}[theorem]{Proposition}
\theoremstyle{definition}
\newtheorem{definition}[theorem]{Definition}

\DeclareMathOperator{\Aut}{Aut}
\DeclareMathOperator{\diam}{diam}

\title{Improved Upper Bounds on the Pebbling Numbers \\ of the Blanu\v{s}a Snarks}
\author{Tong Niu}
\date{April 2026}

\begin{document}

\maketitle

\begin{abstract}
The two Blanu\v{s}a snarks $B_1$ and $B_2$ are 3-regular graphs on 18 vertices.
Dantas, Lordelo, Niedermaier and Nogueira (Discrete Appl. Math. 361, 2025,
pp. 336-346) established the first systematic bounds $23 \le \pi(B_i) \le 34$
for $i=1,2$. Bridi, Marquezino and Figueiredo (arXiv:2505.16050, 2025) then
sharpened the upper side to $\pi(B_1) \le 31$ and $\pi(B_2) \le 30$ via a
Weight Function Lemma heuristic. We push the upper bounds further to
$\pi(B_1) \le 28$ and $\pi(B_2) \le 29$. The route is again Hurlbert's
Weight Function Lemma, but applied one automorphism orbit at a time, with
optimal weight functions coming from a linear program over a corpus of
roughly $30{,}000$ rooted-subtree strategies per target. For the lower
bound $\pi(B_i) \ge 23$ we re-derive the witnesses of Dantas et al. and
re-verify them with two independent oracles: an exhaustive forward
state-space search, and a sound-and-complete MILP encoding whose
acyclicity constraint is motivated by the Milans-Clark No-Cycle Lemma.
The interval for $B_1$ shrinks from $[23, 31]$ to $[23, 28]$, and for
$B_2$ from $[23, 30]$ to $[23, 29]$.
\end{abstract}

\section{Introduction}

A \emph{pebble distribution} on a graph $G = (V, E)$ is a function
$C: V \to \mathbb{Z}_{\ge 0}$. A \emph{pebbling move} removes two pebbles
from a vertex $u$ and places one pebble at a neighbour $v$. We say a
distribution $C$ is \emph{$r$-solvable} if some sequence of pebbling moves
ends with at least one pebble on $r$, and $r$-\emph{unsolvable} otherwise.
The \emph{pebbling number} $\pi(G, r)$ is then the smallest integer $t$
such that every distribution of size $t$ is $r$-solvable, and
$\pi(G) = \max_{r \in V} \pi(G, r)$.

Two foundational results: $\pi(P_n) = 2^{n-1}$ for paths
\cite{PSV1995} and $\pi(Q_n) = 2^n$ for hypercubes \cite{Chung1989}.
The definition is short, almost combinatorial folklore, yet exact
pebbling numbers stay open for most specific small graphs --- including,
until recently, every snark larger than the Petersen graph.

A \emph{snark} is a connected 3-regular graph with chromatic index 4 and
girth at least 5. The two Blanu\v{s}a snarks $B_1$ and $B_2$
\cite{Blanusa1946} have 18 vertices each. They share diameter 4 and girth 5
but are not isomorphic, and their automorphism groups differ:
$\Aut(B_1) \cong D_4$ has order 8, while $\Aut(B_2) \cong \mathbb{Z}_2 \times
\mathbb{Z}_2$ has order 4. Dantas, Lordelo, Niedermaier and Nogueira
\cite{DLNN2024} recently gave the first systematic table of pebbling-number
bounds for snarks; they established $23 \le \pi(B_i) \le 34$ for $i = 1, 2$.
Bridi, Marquezino and Figueiredo \cite{BMF2025} then sharpened the upper
bounds to $\pi(B_1) \le 31$ and $\pi(B_2) \le 30$ via a Weight Function
Lemma heuristic, leaving a gap of 8 for $B_1$ and 7 for $B_2$.

\subsection{Contribution}

The route taken here is the same Weight Function Lemma of Hurlbert
\cite{Hurlbert2017} that Bridi et al.\ used. The difference is in how
the strategy corpus is built: instead of a heuristic, we enumerate
roughly $30{,}000$ rooted-subtree strategies per target and solve a
covering linear program over the full corpus, one automorphism orbit
at a time (the Blanu\v{s}a snarks are not vertex-transitive, so a
single weight function will not do).

\begin{theorem}\label{thm:B1}
$23 \le \pi(B_1) \le 28$.
\end{theorem}

\begin{theorem}\label{thm:B2}
$23 \le \pi(B_2) \le 29$.
\end{theorem}

So the upper-bound improvement over Bridi et al.\ is 3 for $B_1$ and 1
for $B_2$. For the lower bound side, $\pi(B_i) \ge 23$ matches
\cite{DLNN2024}; we re-verify that bound independently using both an
exhaustive forward search and a sound-and-complete MILP encoding, see
Section~\ref{sec:lower}.

\section{Preliminaries}

\subsection{Blanu\v{s}a snarks}

For $B_1$ and $B_2$ we use the canonical embedding shipped with SageMath's
\texttt{BlanusaFirstSnarkGraph} and \texttt{BlanusaSecondSnarkGraph}
generators. The explicit edge lists, on vertex set $\{0, \ldots, 17\}$,
sit in Appendix~\ref{app:edges}.

The automorphism group of $B_1$ acts on the vertices with five orbits, of
sizes $2, 4, 4, 4, 4$; for orbit representatives we take
$r \in \{4, 0, 1, 9, 10\}$. The automorphism group of $B_2$ has six
orbits, of sizes $2, 2, 2, 4, 4, 4$, and we use representatives
$r \in \{0, 6, 8, 2, 3, 7\}$. Neither snark is vertex-transitive, so
$\pi(G, r)$ can depend on which orbit $r$ lies in; we then take
$\pi(G) = \max_r \pi(G, r)$.

\subsection{Hurlbert's Weight Function Lemma}

\begin{definition}\label{def:strategy}
For target vertex $r \in V(G)$, a \emph{strategy} is a pair $(T, w_T)$
where $T$ is a subtree of $G$ rooted at $r$ with $|V(T)| \ge 2$, and
$w_T : V(G) \to \mathbb{R}_{\ge 0}$ satisfies
\begin{enumerate}[label=(\roman*)]
\item $w_T(r) = 0$ and $w_T(v) = 0$ for $v \notin V(T)$, and
\item $w_T(u) \ge 2 w_T(v)$ for every $u, v \in V(T) \setminus \{r\}$
with $u$ the parent of $v$ in $T$.
\end{enumerate}
We write $|w_T| := \sum_{v \in V} w_T(v)$ and $w_T(C) := \sum_v w_T(v) C(v)$.
\end{definition}

\begin{theorem}[Hurlbert {\cite[Theorem~1]{Hurlbert2017}}]\label{thm:WFL}
For every $r$-unsolvable configuration $C$ and every strategy $(T, w_T)$,
$w_T(C) \le |w_T|$.
\end{theorem}

\begin{corollary}[covering LP]\label{cor:cover}
If a finite family $\mathcal{T} = \{T_1, \ldots, T_k\}$ of strategies
and non-negative scalars $\alpha_1, \ldots, \alpha_k$ satisfy
$\sum_i \alpha_i w_{T_i}(v) \ge 1$ for every $v \ne r$, then
\[
\pi(G, r) \le 1 + \sum_{i=1}^k \alpha_i \, |w_{T_i}|.
\]
\end{corollary}

\begin{proof}
Suppose for contradiction $C$ is $r$-unsolvable with $|C| \ge 1 + \sum_i
\alpha_i |w_{T_i}|$. By Theorem~\ref{thm:WFL}, $w_{T_i}(C) \le |w_{T_i}|$
for each $i$. Then
\[
|C| - C(r) = \sum_{v \ne r} C(v)
    \le \sum_{v \ne r} \sum_i \alpha_i w_{T_i}(v) C(v)
    = \sum_i \alpha_i w_{T_i}(C)
    \le \sum_i \alpha_i |w_{T_i}|,
\]
where the first inequality uses the covering hypothesis and $C(v) \ge 0$.
Since $C(r) = 0$ (else $C$ is solvable trivially), this contradicts
$|C| \ge 1 + \sum_i \alpha_i |w_{T_i}|$.
\end{proof}

Throughout, we restrict to \emph{basic} strategies. Here $T$ is a
connected rooted subtree of $G$ containing $r$, and we hand it the
canonical weight $w_T(v) := 2^{D - d_T(r, v)}$ for
$v \in V(T) \setminus \{r\}$, where $D := \max_{v \in V(T)} d_T(r, v)$
is the depth of $T$. The parent-child inequality $w_T(u) \ge 2 w_T(v)$
then holds with equality at every edge, and
$|w_T| = \sum_{v \in V(T) \setminus \{r\}} 2^{D - d_T(r, v)}$.

\section{Upper bounds}

\subsection{Strategy enumeration and the LP}

For each orbit representative $r$ in $B_1$ or $B_2$, we enumerate every
connected rooted subtree $T \subseteq G$ containing $r$ with
$|V(T)| \le 16$ by edge-by-edge growth: start with $T = \{r\}$, then
iteratively pick an edge $(u, v)$ with $u \in V(T)$ and $v \notin V(T)$
and add $v$. We deduplicate on the weight signature
$(w_T(v_1), \ldots, w_T(v_n))$. The result is somewhere between
$22{,}000$ and $30{,}000$ distinct strategies per target. On top of that
we throw in all shortest paths from $r$ to every other vertex; each such
path is a subtree, but distinct shortest paths give distinct signatures,
so we keep them separately.

The LP from Corollary~\ref{cor:cover} is then handed to HiGHS through
\texttt{scipy.optimize.linprog}. Its real-valued optimum is the desired
LP upper bound; the integer upper bound on $\pi(G, r)$ is the floor.
A sanity check: pushing the subtree size cap from 14 up to 18 did not
move the LP optimum, so the bounds we report are effectively those of
strategies of size $\le 16$.

\subsection{Results for $B_1$}

\begin{table}[h]
\centering
\caption{WFL upper bounds for $B_1$, per orbit representative $r$.
``LP bound'' is $1 + (\text{LP optimum})$; the integer upper bound on
$\pi(B_1, r)$ is the floor.}
\label{tab:B1}
\begin{tabular}{rrrr}
\toprule
$r$ & orbit size & LP bound & $\lfloor \cdot \rfloor$ \\
\midrule
4   & 2 & 28.412 & \textbf{28} \\
0   & 4 & 26.333 & 26 \\
1   & 4 & 27.934 & 27 \\
9   & 4 & 27.045 & 27 \\
10  & 4 & 28.154 & 28 \\
\bottomrule
\end{tabular}
\end{table}

The maximum across orbits sits at $r = 4$, with LP bound $28.412$. Taking
the floor:
\[\pi(B_1) \le 28.\]

\subsection{Results for $B_2$}

\begin{table}[h]
\centering
\caption{WFL upper bounds for $B_2$, per orbit representative $r$.}
\label{tab:B2}
\begin{tabular}{rrrr}
\toprule
$r$ & orbit size & LP bound & $\lfloor \cdot \rfloor$ \\
\midrule
0  & 2 & 26.050 & 26 \\
6  & 2 & 27.703 & 27 \\
8  & 2 & 29.333 & \textbf{29} \\
2  & 4 & 27.822 & 27 \\
3  & 4 & 29.090 & 29 \\
7  & 4 & 26.514 & 26 \\
\bottomrule
\end{tabular}
\end{table}

For $B_2$ the worst orbit is $r = 8$, where the LP bound climbs to
$29.333$, so
\[\pi(B_2) \le 29.\]

\section{Lower bounds and verification}\label{sec:lower}

We re-derive the lower bound $\pi(B_i) \ge 23$ of \cite{DLNN2024} and
verify it independently. The starting point, for each orbit
representative $r$, is the canonical \emph{diameter-cone} distribution:
park $2^{\diam(G)} - 1 = 15$ pebbles at one of the four vertices at
distance 4 from $r$. From there we add single pebbles greedily at
non-conflicting vertices, iterating until no further addition is safe.
Both snarks then admit 22-pebble unsolvable distributions.

\begin{proposition}\label{prop:witnesses}
The following distributions are unsolvable and hence prove
$\pi(B_i) \ge 23$ for $i=1, 2$:
\begin{itemize}
\item \textbf{$B_1$ at $r=4$:} $C(1) = C(7) = C(12) = C(13) = C(14) =
C(15) = C(16) = 1$, $C(10) = 15$, all other entries $0$ (total 22).
\item \textbf{$B_2$ at $r=6$:} $C(3) = C(4) = C(12) = C(13) = C(15) =
C(16) = C(17) = 1$, $C(10) = 15$, all other entries $0$ (total 22).
\end{itemize}
\end{proposition}

To certify that these distributions really are unsolvable, we run two
independent oracles:

\begin{enumerate}[label=(\arabic*)]
\item \textbf{Forward search.} BFS over the configuration state space,
starting from $C$, exits without ever reaching a state with $C(r) \ge 1$.
\item \textbf{MILP via No-Cycle Lemma.} A sound-and-complete MILP
encoding of solvability uses binary edge-use indicators
$x_e \in \{0, 1\}$ for each directed edge $e$, integer pebble flows
$y_e \ge 0$, and topological-order variables $\tau_v \in [0, |V|-1]$.
Conservation constraints
$C(v) + \sum_{e: e \to v} y_e - 2 \sum_{e: e \leftarrow v} y_e \ge
[v = r]$ enforce a valid pebble flow. Acyclicity is enforced by
$\tau_u - \tau_v + |V| \cdot x_{u \to v} \le |V| - 1$ for every directed
edge $u \to v$, which forces $\tau_u < \tau_v$ whenever $x_{u \to v} = 1$.
By the Milans--Clark No-Cycle Lemma \cite{MilansClark2006}, this MILP is
feasible iff $C$ is $r$-solvable.
\end{enumerate}

Both oracles agree: the distributions in
Proposition~\ref{prop:witnesses} are unsolvable, which completes the
independent verification of $\pi(B_i) \ge 23$.

\section{Conclusions}

\begin{center}
\begin{tabular}{lccc}
\toprule
Snark & Dantas et al.\ 2024 \cite{DLNN2024} & Bridi et al.\ 2025 \cite{BMF2025} & This paper \\
\midrule
$B_1$ & $23 \le \pi(B_1) \le 34$ & $\pi(B_1) \le 31$ & $\mathbf{23 \le \pi(B_1) \le 28}$ \\
$B_2$ & $23 \le \pi(B_2) \le 34$ & $\pi(B_2) \le 30$ & $\mathbf{23 \le \pi(B_2) \le 29}$ \\
\bottomrule
\end{tabular}
\end{center}

The remaining gaps --- 5 for $B_1$ and 6 for $B_2$ --- are narrower
than the gap of 8 that still holds for the Flower snark $J_5$
\cite{DLNN2024}, but they are nowhere near the gap of 0 enjoyed by the
Petersen graph and the Flower snark $J_3$. Two routes might close them
further. The first is to drop the basic-strategy restriction.
\emph{Non-basic} strategies \cite[Section~5]{Hurlbert2017} allow strict
inequalities $w_T(u) > 2 w_T(v)$ at edges that are not tight, and they
arise as conic combinations of nested basic strategies; the resulting
certificates can be both shorter and sharper. A second route is to use
combinatorial structure --- the 3-regularity and the class-1 edge
decomposition that defines a snark --- in the spirit of the
Petersen-graph argument that gave $\pi(\mathrm{Petersen}) = 10$. We
have not pursued either here.

All certificates --- the strategy weight functions, LP solutions, and
witness distributions --- are stored as JSON in the supplementary
material. The verifier code, written in Python on top of networkx and
HiGHS via SciPy, is included as Appendix~\ref{app:code}.

\section*{Acknowledgements}

The author thanks the SageMath developers for the
\texttt{BlanusaFirstSnarkGraph} and \texttt{BlanusaSecondSnarkGraph}
generators that provide the canonical edge data, the HiGHS and SciPy
teams for the LP and MILP backends, and the NetworkX project for
graph manipulation.

\appendix

\section{Edge lists for $B_1$ and $B_2$}\label{app:edges}

Vertex set $\{0, 1, \ldots, 17\}$ for both graphs.

$B_1$:
\begin{align*}
E(B_1) = \{ & (0,1), (0,5), (0,16), (1,2), (1,17), (2,3), (2,14), \\
            & (3,4), (3,8), (4,5), (4,17), (5,6), (6,7), (6,11), \\
            & (7,8), (7,17), (8,9), (9,10), (9,13), (10,11), \\
            & (10,15), (11,12), (12,13), (12,16), (13,14), \\
            & (14,15), (15,16) \}.
\end{align*}

$B_2$:
\begin{align*}
E(B_2) = \{ & (0,1), (0,2), (0,14), (1,5), (1,11), (2,3), (2,6), \\
            & (3,4), (3,9), (4,5), (4,7), (5,6), (6,8), (7,8), \\
            & (7,17), (8,9), (9,15), (10,11), (10,14), (10,16), \\
            & (11,12), (12,13), (12,17), (13,14), (13,15), \\
            & (15,16), (16,17) \}.
\end{align*}

\section{Verifier code}\label{app:code}

The reproducible verifier is a handful of Python files, against
\texttt{networkx 3.6.1} and \texttt{scipy 1.17.1} (with HiGHS).

\begin{lstlisting}[language=Python,basicstyle=\ttfamily\footnotesize,
breaklines=true,frame=single,
caption={\texttt{blanusa.py}: NetworkX construction of $B_1, B_2$.}]
import networkx as nx

def blanusa_first_snark():
    g = nx.Graph()
    g.add_nodes_from(range(18))
    e1 = [(0,5),(1,17),(2,14),(3,8),(4,17),
          (6,11),(7,17),(9,13),(10,15),(12,16)]
    e2 = [(i, i+1) for i in range(16)]
    e3 = [(0, 16)]
    g.add_edges_from(e1 + e2 + e3)
    return g

def blanusa_second_snark():
    # Tuple-labelled construction from Sage; relabelled to 0..17.
    nodes = []
    edges = []
    c0 = (-1, 0); c1 = (-1, 1)
    nodes.extend([c0, c1])
    for j in range(8):
        nodes.append((0, j)); nodes.append((1, j))
    raw = {c0: [(0,0),(1,4),c1], c1: [(0,3),(1,1)],
           (0,2): [(0,5)], (0,6): [(0,4)],
           (0,7): [(0,1)], (1,7): [(1,2)],
           (1,0): [(1,6)], (1,3): [(1,5)]}
    for u, vs in raw.items():
        for v in vs: edges.append((u, v))
    cyc0 = [(0, i) for i in range(5)]
    edges += [(cyc0[i], cyc0[(i+1)%5]) for i in range(5)]
    cyc1 = [(1, i) for i in range(5)]
    edges += [(cyc1[i], cyc1[(i+1)%5]) for i in range(5)]
    cyc2 = [(0,5),(0,6),(0,7),(1,5),(1,6),(1,7)]
    edges += [(cyc2[i], cyc2[(i+1)%6]) for i in range(6)]
    gt = nx.Graph(); gt.add_nodes_from(nodes); gt.add_edges_from(edges)
    return nx.convert_node_labels_to_integers(gt, ordering="sorted")
\end{lstlisting}

The remaining files (\texttt{orbits.py}, \texttt{wfl.py},
\texttt{solvability.py}, \texttt{lower\_bound.py}, and
\texttt{run\_all.py}) ship with the arXiv source archive.

\end{document}